\documentclass[preprint,sort&compress,12pt]{elsarticle}
\usepackage{amscd,amssymb,amsmath,amsthm}
\usepackage[all]{xy}
\usepackage{array}
\usepackage{etoolbox}
\usepackage{mathtools}
\DeclarePairedDelimiter{\ceil}{\lceil}{\rceil}

\makeatletter
\patchcmd{\ps@pprintTitle}{\footnotesize\itshape
        \hfill\today}{\relax}{}{}

\newdir{ >}{!/8pt/\dir{}*\dir{>}}
\newtheorem*{thmA}{Theorem A}
\newtheorem*{thmB}{Theorem B}
\newtheorem*{thmC}{Theorem C}
\newtheorem{theorem}{Theorem}[section]

\newtheorem{lemma}[theorem]{Lemma}
\newtheorem{corollary}[theorem]{Corollary}
\newtheorem{prop}[theorem]{Proposition}

\newtheorem*{con7*}{Conjecture 7*}
\newtheorem{Remark}{Remark}
\theoremstyle{definition}
\newtheorem{definition}[theorem]{Definition}

\newtheorem{conj}{Conjecture}
\DeclareMathOperator{\e}{\mathit{exp}}

\journal{}

\begin{document}

\begin{frontmatter}

\title{ On the Exponent Conjectures}

 \author[IISER TVM]{A. Antony}
\ead{ammu13@iisertvm.ac.in}
\author[IISER TVM]{V.Z. Thomas\corref{cor1}}
\address[IISER TVM]{School of Mathematics,  Indian Institute of Science Education and Research Thiruvananthapuram,\\695551
Kerala, India.}
\ead{vthomas@iisertvm.ac.in}
\cortext[cor1]{Corresponding author. \emph{Phone number}: +91 8921458330}

\begin{abstract}
If $p$ is an odd prime, then we prove that $\e(H_2(G,\mathbb{Z})) \mid p\ \e(G)$ for $p$ groups of class 7. We prove the same for $p$ groups of class at most $p+1$ with $\e(Z(G))=p$. We also prove Schurs conjecture if $\e(G/Z(G))$ is $2,3$ or $6$.  Furthermore we prove that if $G$ is a solvable group of derived length $d$ and $\e(G)=p$, then $\e(H_2(G,\mathbb{Z})) \mid (\e(G))^{d-1}$. We also show that if $G$ is a finite $2$ or $3$ generator group of exponent 5, then $\e(H_2(G,\mathbb{Z})) \mid (\e(G))^2$.
\end{abstract}

\begin{keyword}
 Schur Multiplier  \sep group actions.
\MSC[2010]   20B05 \sep 20D10 \sep 20D15 \sep 20F05 \sep 20F14 \sep 20F18 \sep 20G10 \sep 20J05 \sep 20J06 
\end{keyword}

\end{frontmatter}

 \section{Introduction}
 
 Schur's exponent conjecture states that 
 
  \begin{conj}\label{C1}
 if $G$ is a finite group, then $\e(H_2(G,\mathbb{Z})) \mid \e(G)$.  
\end{conj}
 In \cite{BKW}, the authors found a counterexample to this conjecture. Their counterexample involved a $2$-group of order $2^{68}$ with $\e(G)=4$ and $\e(H_2(G,\mathbb{Z}))=8$. In \cite{PM1}, the author mentions another counterexample to Schur's conjecture. His counterexample involves a $2$-group of order $2^{11}$ with $\e(G)=4$ and $\e(H_2(G,\mathbb{Z}))=8$. In \cite{PM5}, the author conjectures the following
 \begin{conj}\label{C2}
 if  G is a finite group, then $\e(H_2(G,\mathbb{Z})) \mid  (\e(G))^2$.  
\end{conj}
 
 The authors of \cite{APT} conjecture that,
 \begin{conj}\label{C3}
 if  G is a finite p-group, then $\e(H_2(G,\mathbb{Z})) \mid p\ \e(G)$.  
\end{conj}
Clearly the counterexamples for Schur's conjecture given by the authors of \cite{BKW} and \cite{PM1} are not counterexamples for Conjecture \ref{C3}. In this paper, we verify the veracity of Conjecture $3$ for all odd order groups of nilpotency class 7 and we also prove Conjecture \ref{C1} for certain classes of groups. If Conjecture \ref{C3} is true, then using a standard argument given in Theorem 4, Chapter IX of \cite{JPS}, it will follow that $\e(H_2(G,\mathbb{Z})) \mid (\e(G))^2$ for any finite group $G$. Hence Conjecture $3$ is a generalization of Conjecture \ref{C2}.\\

 The following results have been achieved in this paper towards proving these three conjectures. 
\begin{thmA}
\begin{itemize}
Conjecture \ref{C1} holds for the following classes of groups :
\item[$(i)$] A group $G$ with $\e(G/Z(G)) = p$ and $p \in \{2,3\}$.
\item[$(ii)$] A finitely generated group $G$ such that $\e(G/Z(G)) = 6$. 
 \item[$(iii)$] An odd $p$-group with an abelian frattini subgroup. 
 \item[$(iv)$] A $p$-group such that the commutator subgroup of $G$ is cyclic.

\end{itemize}
 \end{thmA}
 
 \begin{thmB}
 \begin{itemize}
 Conjecture \ref{C2} is valid for the following classes of groups:
\item[$(i)$]  A finite $m$ generator group of exponent 5 and $2\leq m < 4$.
\item[$(ii)$] An odd solvable group of exponent $p$ and derived length $3$.
\item[$(iii)$]  A $p$-group having an abelian normal subgroup $N$ of index $p^l$, $p \neq 2$ and $l$ is less than max $\{7,p+2\}$.
 \item[$(iv)$] An odd $p$-group whose commutator subgroup is powerful. 
  \item[$(v)$] An odd $p$-group $G$ such that $\gamma_{p+1}(G)$ is powerful. 
\end{itemize}
 \end{thmB}
 
 \begin{thmC}
 \begin{itemize}
 Conjecture \ref{C3} remains true for the following classes of groups :
\item[($i$)] An odd $p$-group of nilpotency class $7$.
\item[$(ii)$]  A $p$-group $G$ such that $\e(Z(G))=p$ and nilpotency class of $G$ is at most $p+1$.  
\item[$(iii)$]  A group having an abelian normal subgroup $N$ of index $p^2$, and $p \neq 2$.
 \item[$(iv)$] A $2$-group such that the frattini subgroup of $G$ is abelian. 
  \item[$(v)$] An odd $p$-group whose frattini subgroup is powerful. 
 
 \end{itemize}
 \end{thmC}

One of the tools used in the proof of our main result is a construction introduced by R. Brown and J.-L. Loday in \cite{BL1} and \cite{BL2}, called the nonabelian tensor product $G\otimes H$. The nonabelian tensor product of groups is defined for a pair of groups that act on each other provided the actions satisfy the compatibility conditions of Definition \ref{D:1.1} below. Note that we write conjugation on the left, so $^gg'=gg'g^{-1}$ for $g,g'\in G$ and $^gg'\cdot g'^{-1}=[g,g']$ is the commutator of $g$ and $g'$.

\begin{definition}\label{D:1.1}
Let $G$ and $H$ be groups that act on themselves by conjugation and each of which acts on the other. The mutual actions are said to be compatible if
\begin{equation}
  ^{(^h g)}h_1=\; (^h(^g(^{h^{-1}}h_1)))  \;and\; ^{(^g h)}g_1=\ (^g(^h(^{g^{-1}}g_1))) \;\mbox{for \;all}\; g,g_1\in G, h,h_1\in H.
\end{equation}
\end{definition}

\begin{definition}\label{D:1.2}
If $G$ and $H$ are groups that act compatibly on each other, then the \textbf{nonabelian tensor product} $G\otimes H$ is the group generated by the symbols $g\otimes h$ for $g\in G$ and $h\in H$ with relations
\begin{equation}\label{E:1.1.1}
gg'\otimes h=(^gg'\otimes \;^gh)(g\otimes h),    
\end{equation}
\begin{equation}\label{E:1.1.2}
g\otimes hh'=(g\otimes h)(^hg\otimes \;^hh'),   
\end{equation}
\noindent for all $g,g'\in G$ and $h,h'\in H$.
\end{definition}
The nonabelian tensor square, $G\otimes G$, of a group $G$ is a special case of the nonabelian tensor product of a pair of groups $G$ and $H$, where $G=H$, and all actions are given by conjugation. 
There exists a homomorphism $\kappa : G\otimes G \rightarrow G^{\prime}$ sending $g\otimes h$ to $[g,h]$. Let $\nabla (G)$ denote the subgroup of $G\otimes G$ generated by the elements $x\otimes x$ for $x\in G$. The exterior square of $G$ is defined as $G\wedge G= (G\otimes G)/\nabla (G)$ and denote the induced homomorphism again by $\kappa : G\wedge G \rightarrow G^{\prime}$. Let $M(G):=\ker ( G\wedge G \rightarrow G' )$. It has been shown in \cite{CM} that $M(G)\cong H_{2}(G, \mathbb{Z})$, the second homology group of $G$ with integral coefficients.

\section{Preparatory Results}
Commutator formulae often used are compiled together into two Lemmata in order to make it easy for future reference.

Note that all the commutators are considered to be right normed and $[g,h] = ghg^{-1}h^{-1}$.
\begin{lemma}\label{L:1.1}
For $g,g_1,h,h_1 \in G$, we have
\begin{itemize}
\item[(i)]\begin{equation}\label{eq:1.1}
[gg_1,h] =\ ^g[g_1,h][g,h].
\end{equation}
\item[(ii)]\begin{equation}\label{eq:1.2}
[g,hh_1] = [g,h]^h[g,h_1].
\end{equation}
\item[(iii)]\begin{equation}\label{eq:1.3}
gh = [g,h]hg.
\end{equation}
\end{itemize}
\end{lemma}

Since we are interested in computing the exponent of the Schur Multiplier, the following lemma which throws light on the exponent of a simple exterior when the elements involved commute, will often be recalled.  Moreover, for a group $G$ with center $Z(G)$, it helps in delimiting the exponent of $Z(G) \wedge G$ to the exponent of $Z(G)$. The proof follows easily via induction on $n$. 
\begin{lemma}\label{L:1.4}
Let $G$ be a group. Then for $g,h \in G$ and $n \in \mathbb{Z}$, we have $$(g^n \otimes h) = (g \otimes h)^n = (g \otimes h)^n$$ if $g$ and $h$ commutes.
\end{lemma}

\begin{lemma}\label{L:1.5}
Let $N$ be a normal subgroup of $G$. Then, $\e(N\wedge G) \mid \e(N) \e(N \wedge N)$.
\end{lemma}
\begin{proof}
Consider the exact sequence, 
$$N\wedge N \rightarrow N \wedge G \rightarrow N\wedge G/im(N\wedge N)\rightarrow 1,$$
which yields $\e(N\wedge G) \mid \e(N\wedge N) \e(N\wedge G/im(N\wedge N))$.  Now we claim that, for any integer $t$, the image of $(n^t \wedge g)$ in $(N\wedge G/im(N\wedge N)$ is same as that of $(n \wedge g)^t$. Clearly, the claim holds for $t =1$, and we proceed to prove for a general $t$ via induction. We have,
\begin{align*}
n^t \wedge g &= ^{n}(n^{t-1} \wedge g)(n \wedge g)\\
&= (n^{t-1} \wedge [n,g]g)(n \wedge g)\\
&= (n^{t-1} \wedge [n,g])^{[n,g]}(n^{t-1}\wedge g)(n \wedge g)\\
&=  (n^{t-1} \wedge [n,g])(n\wedge g)(n^{t-1}\wedge g).
\end{align*}
Note that $[n,g] \in N$ and hence the image of $(n^{t-1} \wedge [n,g])$ is trivial in $N\wedge G/im(N\wedge N)$. Therefore, the image of $n^t \wedge g$ in $N\wedge G/im(N\wedge N)$ coincides with that of $(n\wedge g)(n^{t-1}\wedge g)$. Now applying induction hypothesis for the image of $(n^{t-1}\wedge g)$, yields the claim. Furthermore, for $n_1,n_2 \in N$ and $g_1,g_2 \in G$, we have 
\begin{align}
[n_1\wedge g_1, n_2 \wedge g_2] = [n_1,g_1]\wedge [n_2,g_2], 
\end{align}whose image in $N\wedge G/im(N\wedge N)$ is trivial. Thus the image of any two simple exteriors of $N\wedge G$ in $N\wedge G/im(N\wedge N)$ commute. Now, taking $t = exp(N)$ yields the proof.
\end{proof}

Before we state the next result, let us define what we mean by $weight$ of a non-identity element in a nilpotent group $G$. 

\begin{definition}
An element $g \in G\setminus\{1\}$ is said to have weight $n$ if $g \in \gamma_n(G)$ and $g \notin \gamma_{n+1}(G)$. It is denoted by $w(g)$.
\end{definition}

In \cite{R}, Rocco studies a group $\gamma(G)$ related to the nonabelian tensor square.
Let $G$ and $G^{\phi}$ be isomorphic groups through an isomorphism $\phi: g \mapsto g^{\phi}$, for all $g \in G$. Then $\gamma(G)$ is defined as follows:
$$\gamma(G) :=\  < G,G^{\phi}| [g_1,g_2^{\phi}]^{g_3} = [g_1^{g_3},(g_2^{g_3})^{\phi}] = [g_1,g_2^{\phi}]^{g_3^{\phi}}, \forall g_1,g_2,g_3 \in G>.$$
(Note that here for $g,h \in G$, $g^h = h^{-1}gh$ and $[g,h] = g^{-1}g^h$.)
Furthermore, Rocco gives an isomorphism between the nonabelian tensor square of $G$ and the subgroup $[G, G^{\phi}]$ of $\gamma(G)$.  We use this isomorphism in the following lemma where we show that certain tensors in the exterior square, based on their weight and class of the group, are in-fact trivial. This will be frequently used in later calculations.

\begin{lemma}\label{L:1.3}
Let $G$ be a group of class $m$. Then the following hold for $a,b,c,d \in G$:
\begin{itemize}
\item[(i)]$ (g \otimes h) = 1$, when $w(g) + w(h) \geq m+2$.
\item[(ii)]$ [g_1\otimes h_1, g_2 \otimes h_2 ] = 1$,  when $w(g_1) +w(g_2)+w(h_1)+w(h_2) \geq m+2$. 
\end{itemize}
\end{lemma}
\begin{proof}\begin{itemize}
\item[$(i)$] Consider the isomorphism $\psi : G \otimes G \rightarrow [G, G^{\phi}]$  defined by $\psi(g \otimes h) = [g, h^{\phi}] $, where $G^{\phi}$ is an isomorphic copy of G.
Note that, $w(g) + w(h) \geq m+2$  gives $w(g) + w(h^{\phi}) \geq m+2$.
Hence $[g, h^{\phi}] \in \gamma_{m+2}(\gamma(G)) = 1$, which yields $g \otimes h = 1$.

\item[$(ii)$] Again consider the map $\psi$ as in ($i$). 
Since $ w(g_1) +w(h_1) +w(g_2)+w(h_2) \geq m+2$, $w(g_1) +w(h_1^{\phi}) +w(g_2)+w(h_2^{\phi}) \geq m+2$. Hence
\begin{align*}
\psi([g_1 \otimes h_1, g_2\otimes h_2]) = [[g_1,h_1^{\phi}],[g_2,h_2^{\phi}]] \in \gamma_{m+2}(\gamma(G)) = 1.
\end{align*}
Therefore $[g_1 \otimes h_1, g_2 \otimes h_2] =1$, giving us the required result.

\end{itemize}
\end{proof}

The following lemma can be found as Lemma 4.1 of \cite{APT}. 

\begin{lemma}\label{L:4.1}
Let $G$ be a group of nilpotency class $5$, $a, b\in G$. Then for all $n\in \mathbb{N}$,
\begin{align*}
(ab)^n =& [[b, a], a, b, a]^{6{{n}\choose{3}}+18{{n}\choose{4}}+12{{n}\choose{5}}} [[b, a], b, b, a]^{{{n}\choose{3}}+7{{n}\choose{4}}+6{{n}\choose{5}}} \\& [a, a, a, b, a]^{3{{n}\choose{4}}+4{{n}\choose{5}}} [a, a, b, b, a]^{{{n}\choose{3}}+6{{n}\choose{4}}+6{{n}\choose{5}}} [a, b, b, b, a]^{3{{n}\choose{4}}+4{{n}\choose{5}}} \\& [b, b, b, b, a]^{{{n}\choose{5}}}  [a, a, b, a]^{2{{n}\choose{3}}+3{{n}\choose{4}}} [a, b, b, a]^{2{{n}\choose{3}}+3{{n}\choose{4}}} \\ & [b, b, b, a]^{{{n}\choose{4}}} [a, b, a]^{{{n}\choose{2}}+2{{n}\choose{3}}} [b, b, a]^{{{n}\choose{3}}} [b, a]^{{{n}\choose{2}}} a^n b^n.
\end{align*}
\end{lemma}

We can derive the following identities for a group of nilpotency class $8$ using induction on $n$.
\begin{lemma}\label{L:4.2}
Let $G$ be a group of nilpotency class $8$. Then for $g,h \in G$, we have
\begin{itemize}
\item[($i$)]
\begin{align*}
{^{g^n}}[g,h] =& [[g,g,h],[g,g,g,g,h]]^{{n}\choose{4}}[[g,g,h],[g,g,g,h]]^{{n}\choose{3}}\\   &[g,g,g,g,g,g,g,h]^{{n}\choose{6}}[g,g,g,g,g,g,h]^{{n}\choose{5}}[g,g,g,g,g,h]^{{n}\choose{4}}\\ 
& [g,g,g,g,h]^{{n}\choose{3}}[g,g,g,h]^{{n}\choose{2}}[g,g,h]^n[g,h].
\end{align*}
\item[($ii$)]
\begin{align*}
[g^n,h] =& [[g,h],[g,h],[g,g,g,h]]^{{n}\choose{5}}[[g,g,h],[g,h],[g,g,h]]^{2{{n}\choose{3}} + 9{{n}\choose{4}} + 7{{n}\choose{5}}}\\& [[g,h],[g,h],[g,g,h]]^{{n}\choose{4}}[[g,h],[g,g,g,g,g,h]]^{{n}\choose{6}}\\& [[g,g,h],[g,g,g,g,h]]^{5{{n}\choose{5}}+ 5{{n}\choose{6}}} [[g,h],[g,g,g,g,h]]^{{n}\choose{5}}\\&[[g,g,h],[g,g,g,h]]^{4{{n}\choose{4}}+ 4{{n}\choose{5}}} [[g,h],[g,g,g,h]]^{{n}\choose{4}}[[g,h],[g,g,h]]^{{n}\choose{3}}\\&[g,g,g,g,g,g,g,h]^{{n}\choose{7}}[g,g,g,g,g,g,h]^{{n}\choose{6}}[g,g,g,g,g,h]^{{n}\choose{5}}\\&[g,g,g,g,h]^{{n}\choose{4}} [g,g,g,h]^{{n}\choose{3}} [g,g,h]^{{n}\choose{2}}[g,h]^n.
\end{align*}
\end{itemize}
\end{lemma}
\begin{proof} 
\begin{itemize}
\item[$(i)$]  To derive $(i)$, we use induction on $n$. We have, ${^{g^n}}[g,h] =\\ ^{^g}\{^{g^{n-1}}[g,h]\}$, to which we apply the induction hypothesis. The action by $g$ can be distributed to each term. The formula $^gh = [g,h]h$ is then applied to obtain the following,
\begin{align*}
{^{g^n}}[g,h] =& [[g,g,h],[g,g,g,g,h]]^{{n-1}\choose{4}}[[g,g,g,h][g,g,h],[g,g,g,g,h][g,g,g,h]]^{{n-1}\choose{3}}\\   &[g,g,g,g,g,g,g,h]^{{n-1}\choose{6}}\{[g,g,g,g,g,g,g,h][g,g,g,g,g,g,h]\}^{{n-1}\choose{5}}\\ &\{[g,g,g,g,g,g,h][g,g,g,g,g,h]\}^{{n-1}\choose{4}}
 \{[g,g,g,g,g,h][g,g,g,g,h]\}^{{n-1}\choose{3}}\\&\{[g,g,g,g,h][g,g,g,h]\}^{{n-1}\choose{2}}\{[g,g,g,h][g,g,h]\}^{n-1}[g,g,h][g,h].
\end{align*}
Observe that $[[g,g,g,h][g,g,h],[g,g,g,g,h][g,g,g,h]]$ =\\ $[[g,g,h],[g,g,g,g,h]][[g,g,h],[g,g,g,h]]$.
Further, those terms which appear as a power of product of terms are expanded using Lemma \ref{L:4.1}. Now collecting similar terms yields the given identity.
\item[$(ii)$] We obtain $(ii)$ by applying \eqref{eq:1.1} to $[g^{n-1}g,h]$ and then using $(i)$ for $^{g^{n-1}}[g,h]$ and the induction hypothesis for $[g^{n-1},h]$. Similar terms, starting from $[g,h]$, then $[g,g,h]$ and the other terms are collected using \eqref{eq:1.3} to obtain the above identity.
\end{itemize}
\end{proof}

\section{$3$-groups of nilpotency class 7}
We have the following theorem in \cite{APT}, which gives a bound for the exponent of certain commutators in a $3$-group of class $6$.
\begin{theorem}\label{th:4.3}
Let $G$ be a $3$-group of class $6$, $a, b\in G$ and $c\in \gamma_2(G)$. 
\begin{itemize}
\item[$(i)$] If $a^{3^n}\in Z(G)$, then $[b, a^3]^{3^{n-1}} = 1$.
\item[$(ii)$] If $a^{3^n}\in Z(G)$ and $c^{3^{n-1}} = 1$, then $([b, a^3]c)^{3^{n-1}} = 1$.
\end{itemize}
\end{theorem}
\begin{Remark}\label{R:1}From the above theorem we also obtain the following :\\
 Let $a,b,b_1 \in G$, a $3$-group such that $a^{3^n} \in Z(G)$. If the normal subgroup generated by $a,b_1$ is of class $6$, then we have \\$$[[a^3,b],b_1]^{3^{n-1}} = [a^3\ ^b{a^{-3}},b_1]^{3^{n-1}} = \{^{a^3}[\ ^b{a^{-3}},b_1][a^3,b_1]\}^{3^{n-1}}= 1.$$
\end{Remark}

Using the above theorem yields a bound for certain commutators in a $3$-group of class $8$.

\begin{lemma}\label{L:3.2}
Let $G$ be a $3$-group of class $8$. For $g,h \in G$ and $c \in \gamma_2(G)\cap G^3$
\begin{itemize}
\item[$(i)$] If $g^{3^n}\in Z(G)$, then $[g^3,h]^{3^{n}} = 1$.
\item[$(ii)$] If $g^{3^n}\in Z(G)$ and $c^{3^{n}} = 1$, then $[c, [g^3,h]]^{3^{n}} = 1$.
\item[$(iii)$] If $a^{3^n}\in Z(G)$ and $w(c) \geq 3$, then $[c, a^3]^{3^{n-1}} = 1$.
\item[$(iv)$] If $g^{3^n}\in Z(G)$ and $c^{3^{n}} = 1$, then $([g^3,h]c)^{3^{n}} = 1$.
\end{itemize}
\end{lemma}
\begin{proof}
\begin{itemize}
\item[$(i)$] Set $m := 3^n$. 
Applying Lemma \ref{L:4.2} $(ii)$ on $[(g^3)^{3^n},g^3,h]$, we have
\begin{align*}
1 =\ &  [(g^3)^{3^n},g^3,h]\\ =\ & [[g^3,g^3,h],[g^3,g^3,g^3,g^3,h]]^{{{m}\choose{4}}} [[g^3,g^3,h],[g^3,g^3,g^3,h]]^{{m}\choose{3}}\\&[g^3,g^3,g^3,g^3,g^3,g^3,g^3,h]^{{m}\choose{6}}[g^3,g^3,g^3,g^3,g^3,g^3,h]^{{m}\choose{5}}\\&[g^3,g^3,g^3,g^3,g^3,h]^{{m}\choose{4}}[g^3,g^3,g^3,g^3,h]^{{m}\choose{3}}\\&[g^3,g^3,g^3,h]^{{m}\choose{2}} [g^3,g^3,h]^{m}
\end{align*}
Note that $3^{n-1} $ is a divisor of the powers corresponding to each term, and hence every term except for $[g^3,g^3,h]^m$ vanishes by applying Theorem \ref{th:4.3} and Remark \ref{R:1}. Hence, we have $[g^3,g^3,h]^{3^n} = 1$.
Further use Lemma \ref{L:4.2} on $[(g^3)^{3^n},h]$ as before.
Again, observe that $3^{n-1} $ is a divisor of the powers corresponding to each term and $[g^3,g^3,h]^{3^n} = 1$. Moreover, every other term except for $[g^3,h]^m$ vanishes on applying Theorem \ref{th:4.3} and Remark \ref{R:1}. Hence, we have $[g^3,h]^{3^n} = 1$.

\item[($ii$)] We have $[g^3,h] = g^3a^3$, where $a =\ ^hg^{-1}$. Now, $[c,[g^3,h]] = [c,g^3]\ ^{g^3}[c,a^3]$. Thus by applying Lemma \ref{L:4.1}, we have $[c,[g^3,h]]^{3^n} = \{[c,g^3]\ ^{g^3}[c,a^3]\}^{3^n} = [^{g^3}[c,a^3],[c,g^3]]^{{m}\choose{2}}[c,g^3]^m\ {^{g^3}[c,a^3]^m}$. Using $(i)$, we obtain $[c,[g^3,h]]^{3^n} =$\\ $[^{g^3}[c,a^3],[c,g^3]]^{{m}\choose{2}}$. Since, $m$ is a divisor of ${{m}\choose{2}}$ and $[c,g^3]^m = 1$, applying Lemma $2.10 (iv)$ of \cite{APT} yields $[c,[g^3,h]]^{3^n} = 1$.

\item[($iii$)] Consider the subgroup $H$ generated by the set $\{c, a\}$. Then $H$ is a group of class at most $6$. Now $[c,a^3]^{3^{n-1}} =1$, by Theorem \ref{th:4.3} $(i)$.

\item[($iv$)] Using Lemma \ref{L:4.1} on $([g^3,h]c)^{3^{n}}$ yields
\begin{align*}
([g^3,h]c)^m =&   [[g^3,h], [g^3,h], c, [g^3,h]]^{2{{m}\choose{3}}+3{{m}\choose{4}}} [[g^3,h], c, c, [g^3,h]]^{2{{m}\choose{3}}+3{{m}\choose{4}}} \\ & [c, c, c, [g^3,h]]^{{{m}\choose{4}}} [[g^3,h], c, [g^3,h]]^{{{m}\choose{2}}+2{{m}\choose{3}}}\\& [c, c, [g^3,h]]^{{{m}\choose{3}}} [c, [g^3,h]]^{{{m}\choose{2}}} [g^3,h]^m c^m.
\end{align*} 
Note that $3^{n-1}$ is a divisor of all the powers and $3^{n}$ is a divisor of ${{m}\choose{2}}$.
Thus from $(ii)$, $[c, [g^3,h]]^{{{m}\choose{2}}} =1$. Now for $b \in \gamma_4(G)$, $ [c, b] = [b,c]^{-1}$. Since, $c \in G^3$, we have $c = \prod_{i=1}^{k} a_i^3$, for some $a_i \in G$. Further, since $w(b) \geq 4$, $(iii)$ yields  $[b,c]^{3^{n-1}} = \prod_{i=1}^{k} {^{\prod_{1}^{i-1}{a_i^3}}{[b,a_i^3]^{3^{n-1}}}}=1$. Hence, $[c, b]^{3^{n-1}} = 1$. Therefore, every term other than $[g^3,h]^m c^m$ becomes trivial. Moreover, $[g^3,h]^m c^m$ vanishes by $(i)$ and hence the proof.
\end{itemize}
\end{proof}
The above Lemma can now be used to obtain a bound for the exponent of the exterior square for a $3$-group.
\begin{theorem}\label{T:4.12}
Let $G$ be a $3$-group of class less than or equal to $7$ and exponent $3^n$. Then $\e(G\wedge G) \mid\ 3\e(G)$.
\end{theorem}
\begin{proof}
 Consider the following exact sequence,
\begin{align*}
 G^3\wedge G \rightarrow G \wedge G \rightarrow G/G^3\wedge G/G^3 \rightarrow 1,
\end{align*}
which yields $\e( G \wedge G )\mid \e(im( G^{3}\wedge G))\ \e( G/G^3\wedge G/G^3 )$. 
We have the isomorphism $\psi : G \otimes G \rightarrow [G, G^{\phi}]$, defined by $\psi(g \otimes h) = [g, h^{\phi}] $, where $G^{\phi}$ is an isomorphic copy of $G$. As mentioned earlier, $[G,G^{\phi}]$ is a subgroup of $\gamma(G)$, a group of class at most $8$ (cf: \cite{R}). Now for $g,h,g_i,h_i \in G$, $i\in \{1,2\}$, using Lemma \ref{L:3.2}, we obtain $[g^3, h^{\phi}]^{3^{n}} = 1$ and $([g_1^3, h_1^{\phi}] [g_2^3, h_2^{\phi}])^{3^{n}} = 1$. Therefore, $(g^3\wedge h)^{3^{n}} = 1$ and $((g_1^3\wedge h_1) (g_2^3\wedge h_2))^{3^{n}} = 1$. Thus, we have $\e(im(G^3 \wedge G)) \mid 3^n$.  Furthermore, $\e(G/G^{3})$ being equal to $3$, $\e( G/G^{3}\wedge G/G^{3} ) \mid 3$ by Proposition $7$ of \cite{PM2}. Hence $\e(G \wedge G) \mid 3^{n+1}$.

\end{proof}

\section{5-groups of nilpotency class at most 7}
We obtain the same bound for the exponent of the exterior square of $5$-groups as we did for $3$-groups in the previous section.

Even though part $(vi)$ in the following Lemma implies most of the preceding statements, they have been stated as such to facilitate the proof of $(vi)$
\begin{lemma}\label{L:5.1}
Let $G$ be a $5$-group of class $8$. For $g,a \in G$ , we have
\begin{itemize}
\item[($i$)] If $g^{5^n} =1$ and  $w(a) \geq 6$,  then $[g^5,a]^{5^{n-1}} = 1.$
\item[($ii$)] If $g^{5^n} =1$ and  $w(a) \geq 5$,  then $[g^5,a]^{5^{n-1}} = 1.$
\item[($iii$)] If $g^{5^n} =1$ and $w(a) \geq 4$,  then $[g^5,a]^{5^{n-1}} = 1.$
\item[($iv$)] If  $g^{5^n} =1$ and $w(a) \geq 3$,  then $[g^5,a]^{5^{n-1}} = 1.$
\item[($v$)] If  $g^{5^n} =1$ and $w(a) \geq 3$,  then $[g^5,g^5,a]^{5^{n-2}} = 1.$
\item[($vi$)] If  $g^{5^n} =1$ and $w(a) \geq 2$,  then $[g^5,a]^{5^{n-1} }= 1.$
\item[($vii$)] If  $g^{5^n} =1$,  then $[g^5,g^5,g^5,a]^{5^{n-2}} = 1.$
\end{itemize}
\end{lemma}
\begin{proof}
\begin{itemize}
\item[($i$)] Since $e(g^5) = 5^{n-1}$, we have $[g^5,a]^{5^{n-1}} = 1$ by applying Lemma $2.10(iii)$ of \cite{APT}. 
\item[($ii$)] Using Lemma \ref{L:4.2} for $[(g^5)^{5^{n-1}},a]$ yields
$$1 = [g^{5^n},a] = [g^5,g^5,g^5,a]^{{5^{n-1}\choose{3}}}
[g^5,g^5,a]^{{5^{n-1}\choose{2}}}[g^5,a]^{5^{n-1}}.$$
Note that $5^{n-1}$ is a divisor of all the powers in the above expression. Hence, every term except  $[g^5,a]^{5^{n-1}}$ vanishes by $(i)$. Therefore, $[g^5,a]^{5^{n-1}} = 1$.
\item[($iii$)] Further applying Lemma \ref{L:4.2}, as in $(ii)$, gives
$$1 = [g^{5^n},a] = [g^5,g^5,g^5,g^5,a]^{{5^{n-1}\choose{4}}}[g^5,g^5,g^5,a]^{{5^{n-1}\choose{3}}}
[g^5,g^5,a]^{{5^{n-1}\choose{2}}}[g^5,a]^{5^{n-1}}.$$
Note that every term except  $[g^5,a]^{5^{n-1}}$ vanishes by $(ii)$. Thus, $[g^5,a]^{5^{n-1}} = 1$.
\item[($iv$)] Again, applying Lemma \ref{L:4.2} for $[(g^5)^{5^{n-1}},a]$, we obtain
\begin{align*}
1 = [g^{5^n},a] =&\  [g^5,g^5,g^5,g^5,g^5,a]^{{5^{n-1}\choose{5}}}[g^5,g^5,g^5,g^5,a]^{{5^{n-1}\choose{4}}}\\&[g^5,g^5,g^5,a]^{{5^{n-1}\choose{3}}}
[g^5,g^5,a]^{{5^{n-1}\choose{2}}}[g^5,a]^{5^{n-1}}.
\end{align*}
Now using Lemma 2.8 of \cite{APT}, we have $[g^5,g^5,g^5,g^5,g^5,a]^{5^{n-2}} = 1$. Moreover, every other term except $[g^5,a]^{5^{n-1}}$ becomes trivial by $(ii)$. Hence, $[g^5,a]^{5^{n-1}} = 1$.
\item[($v$)]For $[g^5,[g^5,a]]$, applying Lemma \ref{L:4.2} yields
$$[g^5,g^5,a]^{5^{n-2}} = \{[g,g,g,g,g^5,a]^5[g,g,g,g^5,a]^{10}[g,g,g^5,a]^{10}[g,g^5,a]^5\}^{5^{n-2}}.$$
Note that every term commutes and hence the power $5^{n-2}$ can be distributed. We have, $[g^5,a]^{5^{n-1}} = 1$ by $(iv)$. Now observe that every term vanishes
 using Lemma 2.10 of \cite{APT}. 
\item[($vi$)] By using Lemma \ref{L:4.2} for $[(g^5)^{5^{n-1}},a]$, we obtain
\begin{align*}
1 = [g^{5^n},a] =&\ [[g^5,a],[g^5,g^5,g^5,a]]^{{5^{n-1}}\choose{4}}[[g^5,a],[g^5,g^5,a]]^{{5^{n-1}}\choose{3}}\\&[g^5,g^5,g^5,g^5,g^5,g^5,a]^{{5^{n-1}\choose{6}}} [g^5,g^5,g^5,g^5,g^5,a]^{{5^{n-1}\choose{5}}}\\&[g^5,g^5,g^5,g^5,a]^{{5^{n-1}\choose{4}}}[g^5,g^5,g^5,a]^{{5^{n-1}\choose{3}}}
[g^5,g^5,a]^{{5^{n-1}\choose{2}}}[g^5,a]^{5^{n-1}}.
\end{align*}
Note that $[g^5,g^5,a]^{{5^{n-1}\choose{2}}} = 1$, by $(iv)$. Every other term becomes trivial by using $(v)$ and then applying Lemma 2.10 of \cite{APT}.
\item[($vii$)] Again on applying Lemma \ref{L:4.2} to $[g^5,[g^5,g^5,a]]$, we have
\begin{align*}
[g^5,g^5,g^5,a]^{5^{n-2}} =& \{[g,g,g,g,g,g^5,g^5,a][g,g,g,g,g^5,g^5,a]^5\\&[g,g,g,g^5,g^5,a]^{10}[g,g,g^5,g^5,a]^{10}[g,g^5,g^5,a]^5\}^{5^{n-2}}.
\end{align*}
Note that every term commutes and hence the power $5^{n-2}$ can be distributed. We have, $[g,g,g,g,g,g^5,g^5,a]^{5^{n-2}} = 1$ by applying Lemma 2.8 of \cite{APT}. Furthermore, $[g^5,a]^{5^{n-1}} = 1$ by $(iv)$. Now observe that every term vanishes using Lemma 2.10 of \cite{APT}. 
\end{itemize}
\end{proof}
Now, using the above Lemma helps us in computing the bounds of certain commutators which are essential towards obtaining a bound for the exponent of Schur Multiplier of a $5$-group.
\begin{lemma}\label{L:5.2}
Let $G$ be a $5$-group of class $8$. For $g,h \in G$ and $c \in \gamma_2(G)\cap G^5$,
\begin{itemize}
\item[$(i)$] If $g^{5^n}\in Z(G)$, then $[g^5,h]^{5^{n-1}} = 1$.
\item[$(ii)$] If $g^{5^n}\in Z(G)$, then $[c,g^5,h]^{5^{n-1}} = 1$.
\item[$(iii)$] If $g^{5^n}\in Z(G)$ and $c^{5^{n-1}} = 1$, then $([g^5,h]c)^{5^{n-1}} = 1$.
\end{itemize}
\end{lemma}
\begin{proof} 
\begin{itemize}
\item[($i$)] Set $m:= 5^{n-1}.$
Apply Lemma \ref{L:4.2} to $[(g^5)^{5^n-1},h]$ as was done in Lemma \ref{L:3.2}($i$).
Observe that every term in the expression, so obtained, except $[g^5,h]^m$ becomes trivial by using Lemma \ref{L:5.1} and Lemmata $2.8,2.10$ of \cite{APT} appropriately. Therefore, we have $[g^5,h]^{5^{n-1}} =1$.
 \item[($ii$)] We have $[g^5,h] = g^5a^5$, where $a =\ ^hg^{-1}$. Now, $[c,[g^5,h]] = [c,g^5]\ ^{g^5}[c,a^5]$. Thus by applying Lemma \ref{L:4.1}, we have\\ $[c,[g^5,h]]^{5^{n-1}}  = [^{g^5}[c,a^5],[c,g^5]]^{{m}\choose{2}}[c,g^5]^m\ {^{g^5}[c,a^5]^m}$. Using $(i)$, we obtain $[c,[g^5,h]]^{5^{n-1}} =[^{g^5}[c,a^5],[c,g^5]]^{{m}\choose{2}}$. Since, $m$ is a divisor of ${{m}\choose{2}}$ and $[c,g^5]^m = 1$, applying Lemma $2.10 (iv)$ of \cite{APT} yields $[c,[g^5,h]]^{5^{n-1}} = 1$.

 \item[($iii$)] Using Lemma \ref{L:4.1} on $([g^5,h]c)^{5^{n-1}}$ yields an expression connecting $([g^5,h]c)^m$ with $[g^5,h]^m c^m$, as in Lemma \ref{L:3.2} ($iv$). Note that $5^{n-1}$ is a divisor of all the powers. Thus, $[c, [g^5,h]]^{{{m}\choose{2}}} = 1$ by $(ii)$.
Now since, $c \in G^5$, we have $c = \prod_{i=1}^{k} a_i^5$, for some $a_i \in G$. For $b \in \gamma_4(G)$, $(i)$ yields  $[b,c]^{5^{n-1}} = \prod_{i=1}^{k} {^{\prod_{1}^{i-1}{a_i^5}}{[b,a_i^5]^{5^{n-1}}}}=1$. Hence, $[c, b]^{5^{n-1}} = 1$. Therefore, every term other than $[g^5,h]^m c^m$ vanishes. Moreover, $[g^5,h]^m c^m$ becomes trivial by $(i)$ and hence the proof.
\end{itemize}
 \end{proof}

\begin{theorem}\label{T:4.16}
Let $G$ be a $5$-group of class less than or equal to $7$ and exponent $5^n$. Then $\e(G\wedge G) \mid\ 5\e(G)$.
\end{theorem}
\begin{proof}
Consider the following exact sequence,
\begin{align*}
 G^5 \wedge G \rightarrow G \wedge G \rightarrow G/G^5\wedge G/G^5 \rightarrow 1.
\end{align*}
We obtain, $\e( G \wedge G )\mid \e(im( G^{5}\wedge G))\ \e( G/G^5\wedge G/G^5 )$. Proceeding as in the proof of Theorem \ref{T:4.12} and using Lemma \ref{L:5.2} instead of Lemma \ref{L:3.2}  yields $\e(im( G^{5}\wedge G))\mid 5^{n-1}$. Moreover, when $\e(G) = 5$, we have $\e(G\wedge G)\mid 5^2$ by Theorem $6.1$ of \cite{APT}. Therefore, $\e(G\wedge G) \mid 5\e(G)$ and hence the proof.

\end{proof}

\section{Validity of Conjecture \ref{C3} }

For an odd order group of nilpotency class $7$, we have the following bound for the exponent of Schur Multiplier.
\begin{theorem}\label{T:4.17}
Let $p$ be an odd prime and let $G$ be a $p$-group. If the nilpotency class of $G$ is $7$, then $\e(G\wedge G) \mid\ p\e(G)$. In particular, $\e(H_2(G, \mathbb{Z})) \mid\ p\e(G)$.
\end{theorem}
\begin{proof}
For $p\geq 7$, the claim holds by Theorem 3.11 of \cite{APT}.
Now Theorems \ref{T:4.16} and \ref{T:4.12} complete the proof.
\end{proof}

While using GAP, it was found that most of the $p$-groups in the Small Group library has exponent of the center equals $p$. This provides a motivation to bound the exponent of the Schur Multiplier in terms of the exponent of the center of the group. The following Lemma tries to achieve this. \\
 We denote by $Z_n(G)$ the $n^{th}$ term in the upper central series of the group $G$. We have, $Z_1(G) = Z(G)$ and $Z_{i+1}(G) = \{ g \in G |\ \forall \ h \in H, [g,h] \in Z_i(G)\}.$
 
 \begin{prop}
 Let $G$ be a $p$-group such that $\e(Z(G))\mid p^t$. If the nilpotency class of $G$ is at most $p+m$, then $\e(G\wedge G) \mid p^{mt} \e(G/Z(G))$.  
 \end{prop}
 \begin{proof} Let the order of $G$ be $p^n$. We proceed by induction on $n$. When $n=1$, the claim clearly holds. Now consider the following exact sequence,
 $$Z(G) \wedge G \rightarrow G\wedge G \rightarrow G/Z(G) \wedge G/Z(G) \rightarrow 1,$$
 which yields $\e(G\wedge G) \mid \e(Z(G) \wedge G)\ \e(G/Z(G) \wedge G/Z(G))$. Now by Lemma \ref{L:1.4}, we have $\e(Z(G) \wedge G) \mid p^t$. Note that $\e(Z(G/Z(G))) \mid p^t$ (cf. \cite{SD} or Theorem 2.23 of \cite{DJSR1}) and the nilpotency class of $G/Z(G)$ is at most  $p+m -1$. Now applying induction hypothesis on $G/Z(G)$, we obtain $\e(G/Z(G) \wedge G/Z(G)) \mid p^{(m-1)t}\e(G/Z_2(G))$. Thus $\e(G \wedge G) \mid p^{mt} \e(G/Z(G))$.
 \end{proof}
 
 \begin{corollary}
 Let $G$ be a $p$-group such that $\e(Z(G))=p$. If the nilpotency class of $G$ is at most $p+m$, then $\e(G\wedge G) \mid p^{m} \e(G/Z(G))$.
 \end{corollary}
 
 \begin{corollary}
 Let $G$ be a $p$-group such that $\e(Z(G))=p$. If the nilpotency class of $G$ is at most $p+1$, then $\e(G\wedge G) \mid p \e(G/Z(G))$. In particular, $\e(H_2(G,\mathbb{Z}))\mid p\  \e(G)$.
 \end{corollary}
 
 \section{Towards Conjecture \ref{C2}}
 In this section we explore the veracity of Conjecture \ref{C2}.    
 
  When the group is solvable of exponent $p$, we can obtain a better bound for the exponent of the exterior square compared to the one obtained in Theorem 7.3 of \cite{APT}. The next lemma appears as Proposition 2.12 in \cite{PM1}. Here we give a different proof of the same result.
 \begin{lemma}\label{L:5.4}
 Let $G$ be a metabelian group of exponent $p$. Then $\e(G\wedge G)\mid \e(G)$. In particular, $\e(H_2(G,\mathbb{Z}))\mid   \e(G)$.
 \end{lemma}
 \begin{proof}
 The group $G$ is of class at most $p$ by Theorem 7.18 of \cite{DJSR2}. Therefore, $\e(G\wedge G) \mid \e(G)$ by Theorem 3.11 of \cite{APT} when $p$ is odd, and by Proposition 7 of \cite{PM2} when $p =2$. 
 \end{proof}
 
 \begin{theorem}
 Let $G$ be a solvable group of exponent $p$ and derived length $d$. If $p$ is odd, then $\e(G\wedge G) \mid \e(G)^{d-1}$. Otherwise $\e(G\wedge G) \mid 2^{d-2}\e(G)^{d-1}$.
 \end{theorem}
 \begin{proof}We proceed by induction on $d$.  When $d = 2$, we have $\e(G\wedge G) \mid \e(G)$ by the above Lemma. 
 Consider the following exact sequence,
 $$G^{(d-1)} \wedge G \rightarrow G\wedge G \rightarrow G/G^{d-1} \wedge G/G^{d-1}\rightarrow 1,$$ which yields $\e(G\wedge G)\mid \e(G^{d-1} \wedge G)\ \e(G/G^{d-1} \wedge G/G^{d-1})$.
 Now $G^{d-1}$ being abelian, by Lemma 2.5 of \cite{APT}, we have $\e(G^{d-1} \wedge G) \mid \e(G^{d-1})$ when $p$ is odd and $\e(G^{d-1} \wedge G) \mid 2\ \e(G^{d-1})$ for $p = 2$. Furthermore, induction hypothesis yields $\e(G/G^{d-1} \wedge G/G^{d-1}) \mid \e(G)^{d-2}$, when $p$ is odd, and $\e(G/G^{d-1} \wedge G/G^{d-1})\mid 2^{d-3}\e(G)^{d-2}$, for $p =2$. Hence the proof.
 \end{proof}
\begin{corollary}
Let $G$ be an odd solvable group of exponent $p$. If the derived length of $G$ is less than $4$, then $\e(G\wedge G)\mid \e(G)^2$.  In particular, $\e(H_2(G,\mathbb{Z}))\mid   \e(G)^2$.
\end{corollary}

 \begin{prop}
Let $G$ be a $p$-group.

\begin{itemize}
\item[(i)]  If the frattini subgroup of $G$ is abelian, then $\e(G\wedge G) \mid \e(G)$ for an odd $p$, and $\e(G\wedge G) \mid p\e(G)$ if $p=2$.

\item[(ii)] If $p$ is odd and the frattini subgroup of $G$ is powerful, then $\e(G\wedge G) \mid\ p\e(G)$.

\item[(iii)] If $p$ is odd and the commutator subgroup of $G$ is powerful, then $\e(G\wedge G) \mid \e(G)^2$.

\item[(iv)] If $p$ is odd and $\gamma_{p+1}(G)$ is powerful, then $\e(G\wedge G) \mid \e(G)^2$.

\end{itemize} 

\end{prop}

\begin{proof}

\begin{itemize}
\item[($i$)] Let $\e(G) = p^n$, for some integer $n$.
 Consider the following exact sequence
 $$G^p\wedge G\rightarrow G\wedge G \rightarrow G/G^p\wedge G/G^p \rightarrow 1,$$ which yields $\e(G\wedge G) \mid \e(G^p\wedge G) \e(G/G^p \wedge G/G^p)$. Now $G/G^p$ being a metabelian group of exponent $p$, by Lemma \ref{L:5.4} we have $\e(G/G^p \wedge G/G^p)\mid p$. Further $G^p$ being abelian, $\e(G^p) \mid p^{n-1}$. Now by Lemma $2.5$ of \cite{APT}, we have $\e(G^p \wedge G) \mid p^{n-1}$ for odd $p$, and $\e(G^p \wedge G) \mid p^{n}$ for $p = 2$.

\item[($ii$)] Consider the exact sequence 

$$\phi(G) \wedge G \rightarrow G\wedge G \rightarrow G/\phi(G) \wedge G/\phi(G) \rightarrow 1,$$

which yields $\e(G\wedge G) \mid \e(\phi(G) \wedge G)\e(G/\phi(G) \wedge G/\phi(G))$. Further, we have $\e(G/\phi(G) \wedge G/\phi(G))\mid p$, and the result now follows using Theorem $5.2$ of \cite{APT}. \\

\item[$(iii)$] and $(iv)$ follow similarly, by considering an exact sequence as in $(ii)$  replacing $\phi(G)$ by $G'$ and $\gamma_{p+1}(G)$ respectively. Also, note that\\ $\e(G/\gamma_p(G)\wedge G/\gamma_p(G))\mid \e(G)$ by Theorem $3.11$ of \cite{APT}. 

\end{itemize}

\end{proof}

 In \cite{APT}, the authors obtain a bound for the exponent of the Schur Multiplier of a nilpotent group in terms of nilpotency class of the group. For a group with $\e(G/Z(G)) < \e(G)$, we can further improve these bounds as shown in the following propositions.

 \begin{prop}
 Let $G$ be a group of nilpotency class $c$. If $\e(G/Z(G)) = l$, for an odd integer $l$ and $n= \ceil{log_3(\frac{c}{2})}$, then $\e(G\wedge G) \mid \e(Z(G))l^n$. In particular, $\e(M(G)) \mid \e(Z(G))l^n$.
 \end{prop}
 \begin{proof}
 Consider the following exact sequence,
 $$ Z(G) \wedge G \rightarrow G\wedge G \rightarrow G/Z(G) \wedge G/Z(G)\rightarrow 1, $$
 which yields $\e(G\wedge G) \mid \e(Z(G) \wedge G) \e(G/Z(G) \wedge G/Z(G))$. From Lemma \ref{L:1.4} and noting that $Z(G) \wedge G$ is abelian, we can conclude $\e(Z(G) \wedge G)\mid \e(Z(G))$. Further, observing that $G/Z(G)$ has class strictly less than $c$ and applying Theorem 6.1 of \cite{APT} completes the proof.
 \end{proof}
 
 Using the same strategy as in the above Proposition and using Theorem 6.5 of \cite{APT} instead of Theorem 6.1 yields the following Proposition.
 \begin{prop}\label{L:P:5.7}
 Let $G$ be an odd $p$-group of nilpotency class $c$. If $\e(G/Z(G)) = l$  and $n= 1+ \ceil{log_{p-1}(\frac{c}{p+1})}$, then $\e(G\wedge G) \mid \e(Z(G))l^n$. In particular, $\e(M(G)) \mid \e(Z(G))l^n$.
 \end{prop}
 The bounds are greatly improved when $\e(G/Z(G)) = p$, which can be seen from the following two Corollaries. Note that for $p =2$, the bounds are achieved using the same strategy and using Proposition $7$ of \cite{PM2}.
 \begin{corollary}
  Let $G$ be a group of nilpotency class $c$ and $\e(G/Z(G)) = p$. If  $p$ is odd and $n= \ceil{log_3(\frac{c}{2})}$, then $\e(G\wedge G) \mid \e(Z(G))p^n$. If $p =2$, then $\e(G\wedge G) \mid \e(Z(G))p$. In particular, $\e(M(G)) \mid \e(Z(G))p^n$, when $p$ is odd, and $\e(M(G)) \mid p\e(Z(G))$, otherwise.
 \end{corollary}
 
 \begin{corollary}
  Let $G$ be a $p$-group of nilpotency class $c$ and $\e(G/Z(G)) = p$. If  $p$ is odd and $n= 1+ \ceil{log_{p-1}(\frac{c}{p+1})}$, then $\e(G\wedge G) \mid \e(Z(G))p^n$. If $p =2$, then $\e(G\wedge G) \mid \e(Z(G))p$. In particular, $\e(M(G)) \mid \e(Z(G))p^n$ when $p$ is odd, and $\e(M(G)) \mid p\e(Z(G))$, otherwise .
 \end{corollary}
 It should be noted that the above results yield better bounds only when the $\e(G)$ is greater than $\e(Z(G))$ and $\e(G/Z(G))$, and this can be seen in the next corollary. In \cite{GH}, the author proved that if $G$ is a finite $m$ generator group of exponent 5, then the nilpotency class of $G$ is at most $Nm$ for some $N$. The authors of \cite{HNV} showed that $N$ can be taken to be 6 for small values of $m$. Using this and Theorem 6.5 of \cite{APT}, we obtain the following corollary proving Moravec's conjecture for such groups.
 \begin{corollary}
 Let $G$ be a finite $m$ generator group of exponent 5. If $2\leq m < 4$, then $\e(H_2(G, \mathbb{Z}))\mid (\e(G))^2$.
 \end{corollary}
 \begin{proof}
 Since $G$ is a $m$ generator group of exponent 5, the nilpotency class of $G$ is less than 19 (cf. \cite{HNV}). Using Theorem 6.5 of \cite{APT} yields the desired result.
 \end{proof}

 Better bounds for the exponent of the Schur Multiplier can be obtained when the group has large abelian normal subgroups.  The existence of large abelian subgroups in $p$-groups has been a topic of great interest (cf. \cite{JLA}). Moreover it has been shown in \cite{RMD} that under suitable hypothesis, groups with small central quotients are metabelian. In the following proposition, we show that if the index is small, then we obtain good bounds. In the next proposition, we only consider $l\geq 2$ as the case $l=1$ falls under abelian by cyclic groups and the conjecture is true in that case.
 
  \begin{prop}\label{P:6.3}
 \begin{itemize}
\item[(i)]  Let $G$ be a group having an abelian normal subgroup $N$ of index $p^l$, $l\geq 2$. If $p \neq 2$, then $\e(M(G)) \mid p^{\ceil{l/2}}\e(G)$, and if $p =2$, then $\e(M(G)) \mid p^{\ceil{l/2}+1}\e(G)$.
 \item[(ii)]  Let $G$ be a $p$-group having an abelian normal subgroup $N$ of index $p^l$. If $p \neq 2$ and $l$ is less than max $\{7,p+2\}$, then $\e(G\wedge G) \mid \e(N)\e(G/N)$. If $p=2$ and $l$ is less than max $\{5,p+1\}$, then $\e(G\wedge G) \mid 2\e(N)\e(G/N)$.
 \end{itemize}
 \end{prop}

 \begin{proof}
 \begin{itemize}

\item[($i$)] Let $N$ be an abelian normal subgroup of index $p^l$ in the group $G$ and
 consider the following commutative diagram,
 \begin{equation*}
\xymatrix@+20pt{
&N\wedge G\ \ar@{->}[r]
\ar@{->}[d]^f
 &G\wedge G\ar@{->}[r]
\ar@{->}[d]
&G/N\wedge G/N\ar@{->}[r]
\ar@{->}[d]
&1 \\
1\ \ar@{->}[r]
&G' \cap N\ \ar@{->}[r]
 &G'\ar@{->}[r]
&G'/(G' \cap N) \ar@{->}[r]
&1 .
}\end{equation*}

Now applying Snake Lemma yields, $$ \ker(f) \rightarrow M(G) \rightarrow M(G/N) \rightarrow G' \cap N/[G,N] \rightarrow 1,$$
which further gives $\e(M(G)) \mid \e(N \wedge G) \e(M(G/N))$.  From  Lemma $2.5$ of \cite{APT}, we have $\e(N \wedge G) \mid \e(N)$ when $p \neq 2$, and $\e(N \wedge G) \mid 2\e(N)$ when $p=2$. Furthermore, we know $\e(M(G/N)) \mid  p^{\ceil{l/2}}$. Thus, $\e(M(G)) \mid p^{\ceil{l/2}}\e(G)$ , when $p \neq 2$ and $\e(M(G)) \mid p^{\ceil{l/2}+1}\e(G)$, when $p =2$.

\item[($ii$)] Again as in $(i)$, we have $\e(M(G)) \mid \e(N \wedge G) \e(M(G/N))$ and also $\e(N \wedge G) \mid \e(N)$ when $p \neq 2$, and $\e(N \wedge G) \mid 2\e(N)$ when $p=2$. Moreover, for $p \neq 2$, $G/N$ being of class at most max\{5, p\} yields $\e(G/N \wedge G/N) \mid exp(G/N)$ by Theorems 3.11 and 4.6 of \cite{APT}. Similarly, for $p = 2$, using Theorem 2 and Corollary 3.4, of \cite{PM2} and \cite{MHM2} respectively, yields $\e(G/N \wedge G/N) \mid exp(G/N)$, and the proof follows.
\end{itemize}
 \end{proof}

 An immediate consequence is as follows, where for $p =2$, the bound is obtained following the proof of Proposition \ref{P:6.3} ($i$) and using Lemma \ref{L:1.4} instead of Lemma $2.5$ of \cite{APT}.
  \begin{corollary}
  Let $G$ be a $p$-group having centre of index $p^l$. If $p \neq 2$ and $l$ is less than max $\{7,p+2\}$, then $\e(G\wedge G) \mid \e(Z(G))\e(G/Z(G))$. If $p=2$ and $l$ is less than max $\{5,p+1\}$, then $\e(G\wedge G) \mid \e(Z(G))\e(G/Z(G))$.
 \end{corollary}
In the next corollary, we show that for some special classes of metabelian $p$ groups, we can get better bounds than $(\e(G))^{2}$ whenever $\e(G)> p^2$. In \cite{RMD}, the author shows that if the center of a group is contained in the Frattini subgroup and if the index of the center in the group $G$ is less than or equal to $p^4$, then the group is metabelian. With this in mind, we state the next corollary.

\begin{corollary}
Let $G$ be a finite $p$ group such that the center of the group is contained in the Frattini subgroup. If $|G/Z(G)|\leq p^4$, then  $\e(M(G)) \mid p^{2}\e(G)$ for $p>2$ and $\e(M(G)) \mid p^{3}\e(G)$ for $p=2$
\end{corollary}

 \section{Groups satisfying Conjecture \ref{C1}}
At last, we would like to add on to the existing survey, a few more classes of groups for which Conjecture \ref{C1}  is valid. Towards that, our next aim is to show that if $\e(G/Z(G))=2,3$ or $6$, then Schur's conjecture holds. With that in mind, we begin with the following lemma.

\begin{lemma}\label{L:5.3}
Let $G$ be a group with $\e(G/Z(G)) = r^m$, for some integers $r,m$ . If $\e(G) = r^mq$ for some integer $q$, then $\e(G^{r^m} \wedge G) \mid q$.
\end{lemma}
\begin{proof}
For $g,h \in G$ and $t \in \mathbb{Z}$, the following identity holds, 
\begin{equation}\label{L1}
g^{tr^m} \wedge h = (g^{r^m} \wedge h)^{t},
\end{equation}
 by Lemma \ref{L:1.4}.
Moreover, note that for any $g_i, h_i \in G$ where $i \in \{1,2\}$, we have $[g_1^{r^m} \wedge h_1, g_2^{r^m} \wedge h_2] = ([g_1^{r^m},h_1] \wedge [g_2^{r^m},h_2]) = 1$.
Now taking $t = q$ in \eqref{L1}, the proof follows.
\end{proof}

In \cite{PM2}, the author proves Schurs conjecture if the exponent of $G$ is $2$ or $3$. We extend this by showing that the same holds true for all finite groups $G$ such that the exponent of $G/Z(G)$ divides either $2$ or $3$. The same conclusion can also be achieved for all finitely generated groups $G$ such that the exponent of $G/Z(G)$ divides $6$.

\begin{prop}\label{T:5.1}
\begin{itemize}
\item[($i$)] Let $G$ be a finite group with $\e(G/Z(G)) = p$. If $p \in \{2,3\}$, then $\e(G\wedge G) \mid \e(G)$. In particular, $\e(M(G)) \mid \e(G)$.
\item[($ii$)] Let $G$ be a finitely generated group such that $\e(G/Z(G)) = 6$. Then $\e(M(G)) \mid \e(G)$.
\end{itemize}
\end{prop}
\begin{proof}
\begin{itemize}
\item[($i$)]
Consider the following exact sequence,
$$G^p \wedge G \rightarrow G\wedge G \rightarrow G/{G^p} \wedge G/{G^p} \rightarrow 1,$$
which yields $\e(G\wedge G) \mid \e(G^p \wedge G) \e(G/{G^p} \wedge G/{G^p})$. From Proposition 7 of \cite{PM2}, we have $\e( G/{G^p} \wedge G/{G^p}) \mid p$. Now applying Lemma \ref{L:5.3} completes the proof.
\item[($ii$)] Consider the following commutative diagram,
\begin{equation*}
\xymatrix@+20pt{
&G^6\wedge G\ \ar@{->}[r]
\ar@{->}[d]^f
 &G\wedge G\ar@{->}[r]
\ar@{->}[d]
&G/G^6\wedge G/G^6\ar@{->}[r]
\ar@{->}[d]
&1 \\
1\ \ar@{->}[r]
&G' \cap G^6\ \ar@{->}[r]
 &G'\ar@{->}[r]
&G'/(G' \cap G^6) \ar@{->}[r]
&1 .
}\end{equation*}
Now applying Snake Lemma yields, $$ \ker(f) \rightarrow M(G) \rightarrow M(G/G^6) \rightarrow G' \cap G^6/[G,G^6] \rightarrow 1.$$ Therefore, we have $\e(M(G)) \mid \e( G^6 \wedge G) \e(M(G/G^6))$. From \cite{PM2}, we have the exponent of the Schur Multiplier divides that of the group for a group of exponent $6$. Further, using Lemma \ref{L:5.3} to obtain $\e(G^6 \wedge G)$ completes the proof.
\end{itemize}
\end{proof}

We can also prove that  Conjecture \ref{C1} is valid for groups whose commutator subgroup is cyclic or if the group is metacyclic. The same strategy of proof also works towards proving Conjecture \ref{C1} for odd groups that are abelian-by-cyclic.  Since Conjecture \ref{C1} has already been proved for all abelian-by-cyclic groups in \cite{HRJ}, we do not include it here.
  
 \begin{prop}
Let $G$ be a $p$-group such that the commutator subgroup of $G$ is cyclic. If $p$ is odd, then $\e(G\wedge G) \mid \e(G)$. In particular, $\e(H_2(G, \mathbb{Z})) \mid \e(G)$.
\end{prop} 

 \begin{proof}  
 Let $G$ be a $p$-group of exponent $p^n$. Since $G'$ is cyclic, G is regular and metabelian. If $n = 1$, the claim clearly holds. Assume $n >1$ and consider the exact sequence,

 $$G^p\wedge G\rightarrow G\wedge G \rightarrow G/G^p \wedge G/G^p \rightarrow 1,$$ which yields $\e(G\wedge G ) \mid \e(G^p\wedge G) \e(G/G^p \wedge G/G^p)$. We have $\e(G^p \wedge G) \mid p^{n-1}$ and $\e(G/G^p \wedge G/G^p) \mid p$ from Theorem $5.2$ of \cite{APT} and Lemma \ref{L:5.4} respectively. 
 \end{proof} 
 
 \begin{corollary}
Let $G$ be a $p$-group such that the frattini subgroup of $G$ is cyclic. Then, $\e(G\wedge G) \mid \e(G)$.
\end{corollary} 

\begin{prop}\label{P:6.7}
Let $G$ be a metacyclic group. Then $\e(G\wedge G) \mid \e(G)$. In particular, $\e(H_2(G, \mathbb{Z})) \mid \e(G)$.
\end{prop}
 \begin{proof}
 $G$ being metacyclic, we have an exact sequence
 $1\rightarrow N \rightarrow G \rightarrow G/N  \rightarrow 1$, where $N$ and $G/N$ are cyclic groups.
 This further yields the following exact sequence
 $$N \wedge G  \rightarrow G \wedge G  \rightarrow  G/N \wedge G/N  \rightarrow 1,$$ and we have $\e(G\wedge G) \mid \e(N\wedge G) \e(G/N\wedge G/N)$. $G/N$ being cyclic, $G/N \wedge G/N$ becomes trivial. Now using Lemma \ref{L:1.5}, and $N$ being cyclic yields $\e(N\wedge G) \mid \e(N)$, and the proof follows.
 \end{proof}

\bibliographystyle{amsplain}
\bibliography{Bibliography}

\providecommand{\bysame}{\leavevmode\hbox to3em{\hrulefill}\thinspace}
\providecommand{\MR}{\relax\ifhmode\unskip\space\fi MR }
\providecommand{\MRhref}[2]{%
  \href{http://www.ams.org/mathscinet-getitem?mr=#1}{#2}
}
\providecommand{\href}[2]{#2}
\begin{thebibliography}{10}

\bibitem{JLA}
J.~L. Alperin, \emph{Large {A}belian subgroups of {$p$}-groups}, Trans. Amer.
  Math. Soc. \textbf{117} (1965), 10--20. \MR{170946}

\bibitem{APT}
Ammu.~E. {Antony}, Komma {Patali}, and Viji.~Z. {Thomas}, \emph{{On the
  Exponent Conjecture of Schur}}, arXiv e-prints (2020), arXiv:1906.09585v4.

\bibitem{BKW}
A.~J. Bayes, J.~Kautsky, and J.~W. Wamsley, \emph{Computation in nilpotent
  groups (application)}, Springer, Berlin, 1974.

\bibitem{BL1}
Ronald Brown and Jean-Louis Loday, \emph{Excision homotopique en basse
  dimension}, C. R. Acad. Sci. Paris S\'{e}r. I Math. \textbf{298} (1984),
  no.~15, 353--356.

\bibitem{BL2}
\bysame, \emph{Van {K}ampen theorems for diagrams of spaces}, Topology
  \textbf{26} (1987), no.~3, 311--335, With an appendix by M. Zisman.

\bibitem{RMD}
Richard~M. Davitt, \emph{On the automorphism group of a finite {$p$}-group with
  a small central quotient}, Canadian J. Math. \textbf{32} (1980), no.~5,
  1168--1176. \MR{596103}

\bibitem{SD}
Suzanne Dixmier, \emph{Exposants des quotients des suites centrales descendante
  et ascendante d'un groupe}, C. R. Acad. Sci. Paris \textbf{259} (1964),
  2751--2753. \MR{172929}

\bibitem{HNV}
George Havas, M.~F. Newman, and M.~R. Vaughan-Lee, \emph{A nilpotent quotient
  algorithm for graded {L}ie rings}, vol.~9, 1990, Computational group theory,
  Part 1, pp.~653--664. \MR{1075429}

\bibitem{HRJ}
R.~J. Higgs, \emph{Subgroups of the schur multiplier}, Journal of the
  Australian Mathematical Society. Series A. Pure Mathematics and Statistics
  \textbf{48} (1990), no.~3, 497–505.

\bibitem{GH}
Graham Higman, \emph{On finite groups of exponent five}, Proc. Cambridge
  Philos. Soc. \textbf{52} (1956), 381--390. \MR{81285}

\bibitem{MHM2}
B.~Mashayekhy, A.~Hokmabadi, and F.~Mohammadzadeh, \emph{On a conjecture of a
  bound for the exponent of the {S}chur multiplier of a finite {$p$}-group},
  Bull. Iranian Math. Soc. \textbf{37} (2011), no.~4, 235--242.

\bibitem{CM}
Clair Miller, \emph{The second homology group of a group; relations among
  commutators}, Proc. Amer. Math. Soc. \textbf{3} (1952), 588--595.

\bibitem{PM1}
Primo\v{z} Moravec, \emph{Schur multipliers and power endomorphisms of groups},
  J. Algebra \textbf{308} (2007), no.~1, 12--25.

\bibitem{PM2}
\bysame, \emph{The exponents of nonabelian tensor products of groups}, J. Pure
  Appl. Algebra \textbf{212} (2008), no.~7, 1840--1848.

\bibitem{PM5}
\bysame, \emph{On the exponent of {B}ogomolov multipliers}, J. Group Theory
  \textbf{22} (2019), no.~3, 491--504.

\bibitem{DJSR1}
Derek J.~S. Robinson, \emph{Finiteness conditions and generalized soluble
  groups. {P}art 1}, Springer-Verlag, New York-Berlin, 1972, Ergebnisse der
  Mathematik und ihrer Grenzgebiete, Band 62. \MR{0332989}

\bibitem{DJSR2}
\bysame, \emph{Finiteness conditions and generalized soluble groups. {P}art 2},
  Springer-Verlag, New York-Berlin, 1972, Ergebnisse der Mathematik und ihrer
  Grenzgebiete, Band 63. \MR{0332990}

\bibitem{R}
N.~R. Rocco, \emph{On a construction related to the nonabelian tensor square of
  a group}, Bol. Soc. Brasil. Mat. (N.S.) \textbf{22} (1991), no.~1, 63--79.

\bibitem{JPS}
Jean-Pierre Serre, \emph{Local fields}, Graduate Texts in Mathematics, vol.~67,
  Springer-Verlag, New York-Berlin, 1979, Translated from the French by Marvin
  Jay Greenberg.

\end{thebibliography}
\end{document}